\documentclass[12pt]{amsart}

\usepackage{hyperref}
\usepackage{color}
\usepackage{pinlabel}
\usepackage{graphicx}
\usepackage{xypic}

\usepackage{amsmath,amsfonts,amssymb}

\theoremstyle{plain}
\newtheorem*{theorem*}{Theorem}
\newtheorem*{lemma*}{Lemma}
\newtheorem*{corollary*}{Corollary}
\newtheorem*{corollary-1}{Corollary 1}
\newtheorem*{corollary-2}{Corollary 2}

\newtheorem*{proposition*}{Proposition}
\newtheorem*{proposition-1}{Proposition 1}
\newtheorem*{proposition-2}{Proposition 2}
\newtheorem*{proposition-3}{Proposition 3}

\newtheorem{conjecture*}{Conjecture}
\newtheorem{theorem}{Theorem}[section]
\newtheorem{lemma}[theorem]{Lemma}

\newtheorem{proposition}[theorem]{Proposition}

\theoremstyle{remark}

\newtheorem*{remark}{Remark}

\newtheorem*{claim}{Claim}

\theoremstyle{definition}

\def\scrg{\mathcal{G}}

\def\scry{\mathcal{Y}}

\def\GG{\mathcal{G}}

 \def\sl{\operatorname{\sl}}   \def\Z{\mathbb{Z}} \def\R{\mathbb{R}} \def\C{\mathbb{C}}
\def\N{\mathbb{N}}  \def\l{\lambda} \def\ll{\langle} \def\rr{\rangle}
    \def\bp{\begin{pmatrix}}
\def\sm{\setminus} \def\ep{\end{pmatrix}} \def\bn{\begin{enumerate}} 
   \def\en{\end{enumerate}}
\def\ba{\begin{array}} \def\ea{\end{array}}

 \def\Im{\operatorname{Im}} 
\def\be{\begin{equation}} \def\ee{\end{equation}}

   \def\eps{\epsilon}

    \def\fr12{\frac{1}{2}} \def\z12{\Z[\fr12]}

\def\ol{\overline}

\def\G{\Gamma}

\def\G{\Gamma}

\def\sl{\operatorname{SL}}
\def\psl{\operatorname{PSL}}

\renewcommand\epsilon{\varepsilon}
\DeclareMathAlphabet{\mathbf}{OML}{cmm}{b}{it}

\numberwithin{equation}{section}

\begin{document}

\title{Centralizers in $3$-manifold groups}

\author{Stefan Friedl}
\address{Mathematisches Institut\\ Universit\"at zu K\"oln\\   Germany}
\email{sfriedl@gmail.com}

\begin{abstract}
Using the Geometrization Theorem we prove a result on centralizers in fundamental groups of 3-manifolds.
This result had been obtained by Jaco and Shalen and by Johannson using different techniques.
\end{abstract}

\maketitle

\section{Introduction}

In this paper we will study centralizers in fundamental groups of 3-manifolds. By a 3--manifold we will always mean a compact, orientable, connected, irreducible 3-manifold with empty or toroidal boundary.

Let $\pi$ be a group. The \emph{centralizer} of an element $g\in \pi$ is defined to be the subgroup \[ C_{\pi}(g):=\{h\in \pi \, |\, gh=hg \}.\]
Determining centralizers is an important step towards understanding a group.
The goal of this note is to give a new proof of the following theorem.

\begin{theorem}\label{thm:centgen}
Let $N$ be a  3-manifold. We write $\pi=\pi_1(N)$.
Let $g\in \pi$. If $C_{\pi}(g)$ is non-cyclic, then one of the following holds:
\bn
\item there exists a JSJ torus or a boundary torus $T$ and $h\in \pi$ such that $g\in h\pi_1(T)h^{-1}$ and such that
\[ C_{\pi}(g)=h\pi_1(T)h^{-1},\]
\item there exists a Seifert fibered component $M$ and $h\in \pi$ such that
 $g\in h\pi_1(M)h^{-1}$ and such that
\[ C_{\pi}(g)=hC_{\pi_1(M)}(h^{-1}gh)h^{-1}.\]
\en
\end{theorem}

If $N$ is  Seifert fibered, then the theorem holds trivially, and if $N$ is hyperbolic, then  it follows from well-known properties of hyperbolic 3-manifold groups  (we refer to Section \ref{section:hypmfds} for details).
If $N$ is neither Seifert fibered nor hyperbolic, then by the Geometrization Theorem $N$ has a non-trivial JSJ decomposition, in particular $N$ is Haken,
and in that case the theorem was  proved by Jaco and Shalen  \cite[Theorem~VI.1.6]{JS79}  and independently by Johannson \cite[Proposition~32.9]{Jo79}.
In this note we will give an alternative proof of  Theorem \ref{thm:centgen} for 3-manifolds with non-trivial JSJ decomposition using the  Geometrization Theorem proved by Perelman.
Our proof involves  basic facts about fundamental groups of Seifert fibered spaces and hyperbolic 3-manifolds and it consists of a careful study
of the fundamental group of the graph of groups corresponding to the JSJ decomposition.

In order to determine centralizers of 3-manifolds it thus suffices to understand centralizers of Seifert fibered spaces.
For the reader's convenience we recall the results of Jaco--Shalen and Johannson.
Let $N$ be a Seifert fibered 3-manifold with a given Seifert fiber structure. Then there exists a projection map $p\colon N\to B$ where $B$ is the base orbifold. We denote by $B'\to B$ the orientation cover, note that this is either the identity or a 2-fold cover. Following \cite{JS79} we refer to $p^{-1}_*(\pi_1(B'))$ as the \emph{canonical subgroup} of $\pi_1(N)$. If $f$ is a regular fiber of the Seifert fibration, then we refer to the subgroup of $\pi_1(N)$ generated by $f$ as the \emph{fiber subgroup}. Recall that if $N$ is non-spherical, then the fiber subgroup is infinite cyclic and normal. (Note that the fact that the fiber subgroup is normal implies in particular that it is well-defined, and not just up to conjugacy.)

\begin{remark}
Note that the definition of the canonical subgroup and of the fiber subgroup depend on the Seifert fiber structure.
By \cite[Theorem~3.8]{Sc83} (see also \cite{OVZ67} and \cite[II.4.11]{JS79}) a Seifert fibered 3-manifold $N$ admits a unique Seifert fiber structure unless $N$ is either covered by $S^3$, $S^2\times \R$, or the 3-torus,
or $N=S^1\times D^2$ or if $N$ is an $I$-bundle over the torus or the Klein bottle.
\end{remark}

The following theorem, together with Theorem \ref{thm:centgen}, now classifies centralizers of non-spherical 3-manifolds.

\begin{theorem}\label{thm:centsfs} Let  $N$ be  a  non-spherical Seifert fibered 3-manifold  with a given Seifert fiber structure.
Let $g\in \pi=\pi_1(N)$ be a non-trivial element.
Then the following hold:
\bn
\item if $g$ lies in the fiber group, then $C_{\pi}(g)$ equals the canonical subgroup,
\item if $g$ does not lie in the fiber group, then the intersection of $C_{\pi}(g)$ with the canonical subgroup is abelian, in particular $C_{\pi}(g)$ admits an abelian subgroup of index at most two,
\item if $g$ does not lie in the canonical subgroup, then $C_{\pi}(g)$ is infinite cyclic.
\en
\end{theorem}

The first statement is \cite[Proposition~II.4.5]{JS79}. The second and the third statement follow from
\cite[Proposition~II.4.7]{JS79}.
 Using Theorems \ref{thm:centgen} and \ref{thm:centsfs}  one can now immediately obtain results on root structures and the divisibility of elements in 3-manifold groups.
We refer to \cite[Section~4]{AFW11} for details.

Note that given a group $\pi$ and an element $g\in \pi$ the set of conjugacy classes of $g$ is in a canonical bijection to the set $\pi/C_g(\pi)$. We thus obtain the following corollary to Theorem \ref{thm:centgen}.

\begin{theorem}
Let $N$ be a 3--manifold. If $N$ is not a Seifert fibered 3--manifold, then the number of conjugacy classes is infinite for any $g\in \pi_1(N)$.
\end{theorem}

This result was first obtained by de la Harpe and Pr\'eaux \cite{HP07} using different methods. They consider a slightly larger class of 3--manifolds, but extending our approach to the class of 3--manifolds considered in \cite{HP07} poses no problems.
We also refer to \cite{HP07} for an application of this result to the von Neumann algebra $W_\l^*(\pi_1(N))$.

\subsection*{Acknowledgment.} We would like to thank Matthias Aschenbrenner, Pierre de la Harpe, Saul Schleimer, Stephan Tillmann and Henry Wilton for helpful conversations.

\section{Graphs of groups}\label{section:groupgraph}

\noindent
In this section we summarize some basic definitions and facts concerning graphs of groups and their fundamental groups. We refer to \cite{Ba93, Co89, Se80} for missing details.

\subsection{Graphs}
A {\it graph}\/ $\scry$ consists of a set $V=V(\scry)$ of {\it vertices}\/ and a set $E=E(\scry)$ of {\it edges}, and two maps $E\to V\times V\colon e\mapsto (o(e),t(e))$ and $E\to E\colon e\mapsto\ol{e}$, subject to the following condition: for each $e\in E$ we have $\ol{\ol{e}}=e$, $\ol{e}\neq e$, and $o(e)=t(\ol{e})$.
We sometimes also denote $\ol{e}$ by $e^{-1}$.
 Throughout this paper, all graphs are understood to be connected and finite \textup{(}i.e., their vertex sets
 and edge sets are finite\textup{)}.

%

\subsection{The fundamental group of a graph of groups}\label{sec:pi1}
Let $\scry$ be a graph. A \emph{graph $\mathcal G$ of groups based on $\scry$} consists
of families $\{G_v\}_{v\in V(\scry)}$ and $\{G_e\}_{e\in E(\scry)}$ of groups satisfying $G_e=G_{\ol{e}}$ for every $e\in E(\scry)$, together with a family $\{\varphi_e\}_{e\in E(\scry)}$ of monomorphisms $\varphi_e\colon G_e\to G_{t(e)}$ ($e\in E(\scry)$). We will refer to  $\scry$ as the \emph{underlying graph} of $\scrg$.

\medskip

Let $\scrg$ be a graph of groups based on a graph $\scry$. We recall the construction of the fundamental group $G=\pi_1(\scrg)$ of $\scrg$ from \cite[I.5.1]{Se80}. First, consider the \emph{path group $\pi(\scrg)$}
 which is  generated by the groups $G_v$ ($v\in V(\scry)$) and the elements $e\in E(\scry)$ subject to the relations
$$e \varphi_e(g) \ol{e} = \varphi_{\ol{e}}(g) \qquad (e\in E(\scry), g\in G_e).$$
By a \emph{path in $\scry$ } from a vertex $v$ to a vertex $w $ we mean a sequence
$(e_1, e_2, \dots, e_n)$ where $o(e_1)=v, t(e_i)=o(e_{i+1}), i=1,\dots,n-1$ and $t(e_n)=w$.

By a \emph{path  in  $\scrg$} from a vertex $v$ to a vertex $w$ we mean a sequence
$$(g_0, e_1, g_1, e_2, \dots, e_n, g_n),$$ of elements in $E$
where $(e_1,\dots,e_n)$ is a path of length $n$ in $\scry$ from $v$ to $w$ and
where $g_0\in G_v$ and where $g_i\in G_{t(e_i)}$ for $i=1,\dots,n$.
We write $l(\gamma)=n$ and call it the length of $\gamma$.
We say that the path $\gamma$ {\it represents}\/ the element
$$g=g_0 e_1 g_1 e_2 \cdots e_n g_n$$
of $\pi(\scrg)$.

%

\medskip

Let now $w$ be a fixed vertex of $\scry$.
We will refer to a path from $w$ to $w$ as a \emph{loop based at $w$}. The  fundamental group $\pi_1(\scrg,w)$ of $\scrg$ (with base point $w$) is defined to be the subgroup of $\pi(\scrg)$ consisting of elements represented by loops based at $w$.
If $w'\in V(\scry)$ is another base point, and $g$ is an element of $\pi(\scrg)$ represented by a path from $w'$ to $w$, then $\pi_1(\scrg,w')\to\pi_1(\scrg,w)\colon t\mapsto g^{-1}tg$ is an isomorphism. By abuse of notation we write $\pi_1(\scrg)$ to denote $\pi_1(\scrg,w)$ if the particular choice of base point is irrelevant.

Now let $v\in V$. Pick a path $g$ from $v$ to $w$. Then the map $G_v\to \pi_1(\scrg,w)$ given by $t\mapsto g^{-1}tg$ defines a group morphism which is injective (see again  \cite[I.5.2, Corollary~1]{Se80}).
In particular the vertex groups define subgroups of $\pi_1(\scrg,w)$ which are well-defined up to conjugation.
Given a  graph of groups $\GG$ and a base vertex $w$ it is always understood that for each vertex $v$ we picked once and for all a path
from $v$ to $w$.

We will later on make use of the following operations on paths.
Given a path $p$ in $\scrg$ from $v_1$ to $v_2$ we write $o(p)=v_1$ and $t(p)=v_2$.
Given two paths
\[ \ba{rcl} p&:=&(g_0, e_1, g_1, e_2, \dots, e_n, g_n), \mbox{ and }\\
q&:=&(h_0, f_1, h_1, f_2, \dots, f_m, h_m),\ea \]
with $t(p)=o(q)$
we define
\[ p* q:=(g_0, e_1, g_1, e_2, \dots, e_n, g_n\cdot h_0, f_1, h_1, f_2, \dots, f_m, h_m)\]
which is a path from $o(p)$ to $t(q)$.
Furthermore, given a path
\[ p:=(g_0, e_1, g_1, e_2, \dots, e_n, g_n)\]  we define the inverse path to be
\[ p^{-1}:=(g_n^{-1}, \ol{e_n}, \dots,g_1^{-1},\ol{e_1},g_0^{-1}).\]
  Note that $p^{-1}$ is a path from $t(p)$ to $o(p)$.


\subsection{Reduced paths}

A path $(g_0, e_1, g_1, e_2, \dots, e_n, g_n)$ in $\scrg$ is {\it reduced}\/ if it satisfies one of the following conditions:
\begin{enumerate}
\item $n=0$, or
\item  $n>0$ and $g_i\notin \varphi_{e_i}(G_{e_i})$ for each index $i$ such that $e_{i+1}=\overline{e_i}$.
\end{enumerate}
Given $g\in \pi(\GG)$ we define its length $l(g)$ to be the length of a reduced path representing it. Note that this is well-defined (see \cite[p.~4]{Se80}), i.e. any $g$ is represented by a reduced path and the definition is independent of the choice of the reduced path. Also note that
\[ l(g)=\min\{ l(p) \, |\, \mbox{$p$ a path which represents $g$}\}.\]
Note that $l(g)=0$ if and only if $g$ lies in $G_{v}$ for some $v\in V$.


We say that  $s=(g_0, e_1, g_1, e_2, \dots, e_n, g_n)$ is \emph{cyclically reduced} if
$s$ is reduced and if one of the following holds:
\bn
\item $n=0$, or
\item $e_1\ne \ol{e_n}$, or
\item $e_1=\ol{e_n}$ but $g_ng_0$ is not conjugate to an element in $\Im(\varphi_{e_n})$.
\en
Note that  a reduced loop $s=(g_0, e_1, g_1, e_2, \dots, e_n, g_n)$ is cyclically reduced if and only if   the element it represents  has minimal length in its conjugacy class
in the path group $\pi(\scrg)$.

 We say that $g\in \pi_1(\GG,w)$ is \emph{cyclically reduced} if there exists a cyclically reduced loop which represents it.
 It is straightforward to see that  $g$ is cyclically reduced if and only if any reduced loop representing it is cyclically reduced.
 Also note that if $g$ is cyclically reduced, then $l(g^n)=n\cdot l(g)$.

Any element $g$ of the path groups $\pi(\scrg)$ is conjugate in $\pi(\scrg)$ to a cyclically reduced element $s$, we can thus define $cl(g)=l(s)$.
Note that this is independent of the choice of $s$.
Note that if $g$ is cyclically reduced, then a straightforward argument shows that $l(g^n)=n\cdot l(g)$. In particular
given any $g$ we have $cl(g^n)=n\cdot cl(g)$.
\section{Fundamental groups of 3-manifolds}

In the next two sections we cover properties of fundamental groups of hyperbolic 3-manifold groups and of Seifert fibered spaces,
before we return to the study of 3-manifold groups in general.

\subsection{Fundamental groups of hyperbolic 3-manifolds}\label{section:hypmfds}

Let $N$ be a 3-manifold. We say that $N$ is hyperbolic if the interior admits a complete metric of finite volume and constant sectional curvature equal to $-1$.

Throughout this section we write
\[ U:=\left\{ \bp \eps & a \\ 0&\eps \ep \mbox{ with } \eps\in \{-1,1\} \mbox{ and }a\in \C\right\}\subset \sl(2,\C).\]
Note that $U$ is an abelian subgroup of $\sl(2,\C)$. Recall that $A\in \sl(2,\C)$ is called \emph{parabolic} if it is conjugate to an element in $U$.
We say that $A$ is \emph{loxodromic} if $A$ is diagonalizable with eigenvalues $\l,\l^{-1}$ such that $|\l|>1$.
We  recall the following well known proposition.

\begin{proposition}\label{prop:hypbasics}
Let $N$ be a hyperbolic 3-manifold. Then the following hold:
\bn
\item There exists a faithful discrete representation  $\rho\colon \pi_1(N)\to \sl(2,\C)$.
\item Let $g\in \pi_1(N)$, then $\rho(g)$ is either parabolic or loxodromic.
\item An element $g\in \pi_1(N)$ is conjugate to an element in a boundary component if and only if $\rho(g)$ is parabolic.
\item Let $T$ be a boundary torus, then  there exists a matrix $P\in \sl(2,\C)$ such that
$P \rho(\pi_1(T)) P^{-1}\subset U$.
\item Let $g\in \pi_1(N)$. Then $C_g(\pi_1(N))$ is either infinite cyclic or a free abelian group of rank two.
The latter case occurs precisely when $g$ is conjugate to an element in a boundary component $T$ and in that case $C_g(\pi_1(N))$ is a conjugate of $\pi_1(T)$.
\en
\end{proposition}

We include the proof of the proposition for completeness' sake.

\begin{proof}
\bn
\item
A hyperbolic 3-manifold $N$ admits a faithful discrete representation $\pi_1(N)\to \operatorname{Isom}(\Bbb{H}^3)=\psl(2,\C)$.
  Thurston (see \cite[Section~1.6]{Sh02}) showed that this representation lifts to a faithful discrete representation $\pi_1(N)\to \sl(2,\C)$.
\item This follows immediately from considering the Jordan transform of $\rho(g)$ and from the fact that the infinite cyclic group generated by $\rho(g)$ is discrete in $\sl(2,\C)$.
\item This is well-known, see e.g. \cite[p.~115]{Ma07}.
\item This statement follows easily from the fact that $\pi_1(T)\subset \sl(2,\C)$ is a discrete subgroup isomorphic to $\Z^2$.
\item By (1) we can view  $\pi=\pi_1(N)$ as a discrete, torsion-free subgroup of $\sl(2,\C)$.
Note that the centralizer of any non-trivial matrix in $\sl(2,\C)$ is abelian (this can be seen easily using the Jordan normal form of such a matrix).
Now let $g\in \pi\subset \sl(2,\C)$ be non-trivial. Since $\pi$ is  torsion-free and discrete in $\sl(2,\C)$ it follows easily  that
$C_{\pi}(g)$ is in fact either infinite cyclic
or a free abelian group of rank two.
It now follows from \cite[Proposition~5.4.4]{Th79} (see also \cite[Corollary~4.6]{Sc83} for the closed case) that there exists a boundary component $S$ and $h\in \pi_1(N)$ such that
\[ C_{\pi}(g)=h\pi_1(S)h^{-1}.\]
\en
\end{proof}

Given a group $\pi$ we say that an element $g$ is \emph{divisible by an integer $n$} if
there exists an $h\in \pi$ with $g=h^n$. We say $g$ is \emph{infinitely divisible} if $g$ is divisible by infinitely many integers.
The following lemma is an immediate consequences of Proposition \ref{prop:hypbasics} (5).

\begin{lemma}\label{lem:hypdivisible}
Let $\pi\subset \sl(2,\C)$ be a discrete torsion-free group. Then $\pi$ does not contain any non-trivial  elements which are infinitely divisible.
\end{lemma}

Let $\pi$ be a group. We say that a subgroup $H\subset \pi$ is \emph{division closed} if for any $g\in \pi$ and $n>0$ with $g^n\in H$ the element   $g$ already lies in $H$.
The following lemma is an immediate consequence of Proposition \ref{prop:hypbasics} (2) and (5) and from the observation that $A\subset \sl(2,\C)$ is parabolic (respectively loxodromic)
if and only if a non-trivial power of $A$ is parabolic (respectively loxodromic).

\begin{lemma}\label{lem:divisionclosed}
Let $N$ be a 3-manifold  such that the interior of $N$ is a hyperbolic 3-manifold of finite volume.
Let $T$ be a boundary component of $N$. Then $\pi_1(T)\subset \pi_1(N)$ is division closed.
\end{lemma}

Let $\pi$ be a group. We say that a subgroup  $H$ is \emph{malnormal} if $gHg^{-1}\cap H$ is trivial for any $g\not\in H$.
The following lemma is well-known.

\begin{lemma}\label{lem:malnormal} \label{lem:malnormaltwo}
Let $N$ be a hyperbolic 3-manifold.
\bn
\item Let $T$ be a boundary torus. Then $\pi_1(T)\subset \pi_1(N)$ is malnormal.
\item  Let $T_1$ and $T_2$ be distinct boundary tori. Then  for any $g\in \pi_1(N)
$ we have
$\pi_1(T_1)\cap g\pi_1(T_2)g^{-1}=\{e\}$.
\en
\end{lemma}


%
%
%

\subsection{Fundamental groups of Seifert fibered manifolds}


Let $N$ be a Seifert fibered space with regular fiber $c$.
First note that if  $T$ is a boundary torus, then  the Seifert fibration restricted to $T$ induces a product structure.
It follows that   $c\in \pi_1(T)$ and that $c$ is indivisible in $\pi_1(T)\cong \Z^2$.

The following results  summarize the key properties of fundamental groups of Seifert fibered spaces which are relevant to our discussion.

\begin{theorem}\label{thm:sfs}
Let $N$ be a Seifert fibered 3-manifold with regular fiber $c$. Then there exists an $s\in \N$ with the following property:
If $T$ is a boundary component, and if  $g\not\in \pi_1(T)$ but some power of $g$ lies in $\pi_1(T)$, then there exists $d\leq s$ such that $g^d=c$ or $g^d=c^{-1}$.
\end{theorem}


\begin{proof}
Let $N$ be a Seifert fibered 3-manifold with boundary. Let $s$ be the maximum order of a singular fiber of the fibration.
Let $T$ be a boundary component, and let  $g\not\in \pi_1(T)$ such that  some power of $g$ lies in $\pi_1(T)$.
We denote by $p:N\to B$ the projection to the base orbifold. We denote by $b$ the boundary curve of $B$ corresponding to $T$.
Note that $p(g)\not\in \ll b\rr$ but a power of $p(g)$ lies in $\ll b\rr$. It follows easily from \cite[Remark~II.3.1]{JS79}
that $p(g)$ is of finite order. In particular $g$ corresponds to a singular fiber, and then it follows from the definition of $s$ that there exists a $d\leq s$ such that $g^d=c$ or $g^d=c^{-1}$.
\end{proof}

\begin{lemma}\label{lem:centbdy}
Let $N$ be a Seifert fibered 3-manifold with regular fiber $c$ and let $T$ be a boundary component.
Let $g\in \pi_1(T)$ which is not a power of $c$, then $C_g(\pi_1(N))=\pi_1(T)$.
\end{lemma}

\begin{proof}
We denote by $p:N\to B$ the projection to the base orbifold. Note that $p(g)\in \pi_1(B)$ is non-trivial. It follows easily from \cite[Remark~II.3.1]{JS79} that $C_{p(g)}(\pi_1(B))$ is the group
generated by the boundary curve of $N$ corresponding to $T$. It follows easily that  $C_g(\pi_1(N))=\pi_1(T)$.
\end{proof}

The following lemma is also well-known. It can be proved in a similar fashion as Lemma \ref{lem:centbdy} by considering the equivalent problem in the fundamental group of the base manifold.

\begin{lemma}\label{lem:sfsmalnormal}
Let $N$ be a Seifert fibered 3-manifold. Denote by $c\in \pi_1(N)$ the element represented by a regular fiber.
\bn
\item Let $T$ be a boundary torus  and $g\in \pi_1(N)\sm \pi_1(T)$, then  $\pi_1(T)\cap g\pi_1(T)g^{-1}=\ll c \rr$.
\item  Let $T_1$ and $T_2$ be distinct boundary tori. Then  for any $g\in \pi_1(N)
$ we have
$\pi_1(T_1)\cap g\pi_1(T_2)g^{-1}=\ll c \rr$.
\en
\end{lemma}


We conclude with the following lemma.

\begin{lemma}\label{lem:sfsdivisible}
Let $N$ be a non-spherical Seifert fibered manifold. Then $\pi_1(N)$ does not contain non-trivial elements which are infinitely divisible.
\end{lemma}

\begin{proof}
Let $N$ be a  Seifert fibered manifold. Then there exists a finite cover $N'$ which is an $S^1$-bundle over a surface $S$ (see e.g. \cite[p.~391]{He87} for details). We write $\G=\pi_1(S)$, $\pi=\pi_1(N)$ and $\pi'=\pi_1(N')$. If $N$ is non-spherical
then the long exact sequence in homotopy implies that there exists a short exact sequence
\[ 1\to \Z \to \pi'\to \G\to 1.\]
Since $\Z$ and $\G$ are well-known not to admit any non-trivial infinitely divisible elements, it follows easily that $\pi'$ does not admit a non-trivial infinitely divisible element.
We write $n=[\pi:\pi']$.
Since $N$ is non-spherical we know that $\pi$ is torsion-free.  Note that if $g\in \pi$ is non-trivial, then $g^n$ lies in $\pi'$ and it is also non-trivial. It is now easy to see that
$\pi$ can not admit a non-trivial infinitely divisible element either.
\end{proof}

\subsection{3-manifolds and graphs of groups}

In this section we recall the well-known interpretation of 3-manifold groups as the fundamental group of a graph of groups.
Let $N$ be an irreducible, closed, oriented 3-manifold.
Recall that the JSJ tori are a minimal collection $\{T_1,\dots,T_k\}$ of tori such that the complements of the tori are either atoroidal or Seifert fibered.

We denote by $\scrg(N)$ the corresponding JSJ graph,
i.e.  the vertex set $V=V(\scrg)$ of $\scrg$ consists of the set of components of $N$ cut along $T_1,\dots,T_k$ pieces and the set $E=E(\scrg)$ of (unoriented) edges consists of the set of JSJ tori $T_1,\dots,T_k$.
We sometimes denote the JSJ tori by $T_e,e\in E$ and we denote the components of $N$ cut along $\cup_{e\in E} T_e$ by $N_v, v\in V$. We equip each $T_e$ with an orientation, we thus obtain
two canonical embeddings $i_{\pm}$ of $T_e$ into $N$ cut along $T_e$.
We then denote by $o(e)\in V$ the unique vertex with $i_-(T_e)\in N_{i(e)}$ and
we denote by $t(e)\in V$ the unique vertex with $i_+(T_e)\in N_{f(e)}$.

Suppose that $N$ has a non-trivial JSJ decomposition.
Then given a Seifert fibered component $N_v$ of the JSJ decomposition of $N$
we denote by  $c_{v}\in \pi_1(N_v)$ the  group element defined by a corresponding  regular fiber.
Note that $c_v$ is well-defined up to inversion (see \cite[Lemma~1]{To78} or \cite{GH75}).

We conclude this section with the following theorem.

\begin{theorem}\label{thm:adjsfs}
Let $N$ be a closed, oriented 3-manifold.  Denote by $\GG=\GG(N)$ the corresponding JSJ graph. If $e$ is an edge such that $o(e)$ and $t(e)$ correspond to  Seifert fibered spaces,
then $\varphi_{e}^{-1}(c_{t(e)})\ne c_{o(e)}^{\pm 1}$.
\end{theorem}

\begin{proof}
If  $\varphi_{e}^{-1}(c_{t(e)})$ was equal to $c_{o(e)}^{\pm 1}$, then $N_{o(e)}$ and $N_{t(e)}$ would have Seifert fiber structures which (after an isotopy) match along the edge torus.
But this contradicts the minimality of the JSJ decomposition.
\end{proof}

\section{Proof of the main results}

\subsection{Divisibility in 3-manifold groups}

We will first prove the following theorem.

\begin{theorem}\label{thm:divisible}
Let $N$ be a 3-manifold. If $N$ is not spherical, then $\pi_1(N)$ does not contain any non-trivial elements which are infinitely divisible.
\end{theorem}

\begin{proof}
Let $N$ be a non-spherical 3-manifold and let $x\in \pi_1(N)$ be a non-trivial element.
Since the statement of theorem is independent of the choice of base point and conjugation we can without loss of generality assume that $l(x)=cl(x)$. We write  $l=l(x)$.

First suppose that $l>0$.
Suppose we have $y\in \pi_1(N)$ and $n$ such that $y^n=x$.
Note that $0<cl(x)=cl(y^n)=n\cdot cl(y)$. It now follows immediately
that $n\leq l=cl(x)$.

Now suppose that $l=0$. Note that this means that $x$  lies  in a vertex group $\pi_1(N_w)$.
We now define
\[ d:= \max\{ n\in \N\, |\, x=y^n\mbox{ for some }y\in \pi_1(N_w)\}.\]
Note that $d<\infty$ by  Lemmas \ref{lem:hypdivisible} and \ref{lem:sfsdivisible}.
Furthermore, given a Seifert fibered component $N_v$ we define
\[ s_v:= \mbox{maximum  of the  orders of the  singular fibers of $N_v$}. \]
Finally we define $s$ to be the maximum over all $s_v$. If there are no Seifert fibered components, then  we set $s=1$.
The following claim now implies the theorem.

\begin{claim}
If there exists $y\in \pi_1(N)$ and $n\in \N$ with $y^n=x$, then $n\leq ds$.
\end{claim}

Suppose we have $y\in \pi_1(N)$ and $n$ such that $y^n=x$.
Note that $0=l(x)=cl(x)=cl(y^n)=n\cdot cl(y)$. It now follows that $cl(y)=0$.
If $l(y)=0$, then $y\in \pi_1(N_w)$, hence the conclusion holds trivially by the definition of $d$.
Now suppose that $l(y)>0$. Then there exists a reduced path $p=(g_0,e_1,g_1,\dots,e_l,g_l)$ from $w$ to a vertex $v$
and $z\in \pi_1(N_v)$ such that $y$ is represented by $p*z*p^{-1}$. Among all such pairs $(p,z)$ we  pick a pair which minimizes the length of $p$.

Since $p$ is minimal and $l(p)>0$ we see that $g_lzg_l^{-1}\not\in \Im(\varphi_{e_l})$.
On the other hand  $p*z^n*p^{-1}$ represents $y^n=x$, hence this path is reduced, which implies that $g_lz^ng_l^{-1}\in \Im(\varphi_{e_l})$. It follows that $ \Im(\varphi_{e_l})$ is not division closed,
using  Lemma \ref{lem:divisionclosed} we conclude that $N_v$ is Seifert fibered.

We denote by $c_v$ the regular fiber of $N_v$.
 Recall that by Theorem \ref{thm:sfs}   there exists $r|s_v$ with  $g_lz^{r}g_l^{-1}=c_v$.
It also follows from Theorem \ref{thm:sfs}
that $g_lz^{n}g_l^{-1}=c_v^m\in \Im(\varphi_{e_l})$ for some $m$.
Note that $n=mr$.

We can now apply Lemmas \ref{lem:malnormal} and \ref{lem:sfsmalnormal}, Theorem \ref{thm:adjsfs} and the fact that  $p$ is reduced to conclude that
\[  (g_0,e_1,g_1,\dots,e_{l-1},g_{l-1}\varphi_{e_l}^{-1}(c_v^m)g_{l-1}^{-1},e_{l-1}^{-1},\dots,g_1^{-1},e_1^{-1},g_0^{-1})
\]
is reduced. It follows that $l=1$.
Note that
\[x=g_0\varphi_{e_1}^{-1}(c_v^m)g_0^{-1}=\big(g_0\varphi_{e_1}^{-1}(c_v)g_0^{-1}\big)^m.\]
It follows that $m\leq d$. We also have $r\leq s_v\leq s$. We now conclude that $n=mr\leq ds$.

\end{proof}

\subsection{Commuting elements in 3-manifold groups}

\begin{theorem}\label{thm:commute}
Let $N$ be a 3-manifold. Let $x,y\in \pi_1(N)$ with $x=yxy^{-1}$.
Then one of the following holds:
\bn
\item $x$ and $y$ generate a cyclic group in $\pi_1(N)$, or
\item there exists a JSJ torus $T$ such that $x$ and $y$ lie in a conjugate of $\pi_1(T)\subset \pi_1(N)$, or
\item there exists a Seifert fibered component $M$ of the JSJ decomposition  such that $x$ and $y$ lie in a conjugate of $\pi_1(M)\subset \pi_1(N)$.
\en
\end{theorem}

\begin{proof}
Let $N$ be a 3-manifold. Denote by $\GG=\GG(N)$ the corresponding JSJ graph with vertex set $V$ and edge set $E$. We  denote by $w\in V$ the vertex which contains the base point of $N$.
We denote the  vertex groups by $G_v=\pi_1(N_v), v\in V$.

The theorem holds trivially for Seifert fibered spaces, we  can therefore assume that $N$ is not a Seifert fibered space, in particular that $N$ is not spherical.
Suppose we have $x,y\in \pi_1(N)$ with $x=yxy^{-1}$.
By the symmetry of $x$ and $y$ we can without loss of generality assume that $cl(x)\leq cl(y)$.
Note that the statement of the theorem does not change under conjugation and change of base point,
we can therefore without loss of generality  assume that $cl(x)=l(x)$.

We  represent $y$ by a reduced loop
$p=(h_0,f_1,h_1,\dots,f_{l-1},h_{l-1},f_l,h_l)$ based at $w$.
If $l=0$, then $l(x)=0$ as well since $l(x)=cl(x)\leq cl(y)\leq l(y)=0$. In that case we  are  done by Proposition  \ref{prop:hypbasics} (5). We thus henceforth only consider the case that $l\geq 1$.

 After conjugating $x$ and $y$ with $h_l$ we can without loss of generality assume that $h_l=1$.
Recall that $p$ being reduced implies that for $i=2,\dots,l$ the following holds:
\be \label{equ:or} \mbox{ $f_i\ne \ol{f_{i-1}}$ or $f_i=\ol{f_{i-1}}$ and $h_{i-1}\not\in \Im(\varphi_{f_{i-1}})$.}\ee
We first study the case that $l(x)=0$, i.e. $x\in G_w$.
Clearly we can assume that  $x$ is non-trivial.

Now consider
\[ p*x*p^{-1}=(h_0,f_1,h_1,\dots,f_l,x,f_l^{-1},\dots,h_1^{-1},f_1^{-1},h_0^{-1}).\]
This path is not reduced since $yxy^{-1}$ can be represented by a path of length zero.
It follows that $x \in \Im(\varphi_{f_l})$.
We can now represent $x=yxy^{-1}$ by the following path:
\be\label{equ:pathnow} (h_0,f_1,h_1,\dots,f_{l-1},h_{l-1}\varphi_{f_l}^{-1}(x)h_{l-1}^{-1},f_{l-1}^{-1},\dots,h_1^{-1},f_1^{-1},h_0^{-1}).\ee

\noindent \emph{Case 1:} $l=1$, i.e. $y=(h_0,f_1,1)$.
In that case $yxy^{-1}=x$ is represented by $h_0\varphi_{f_1}^{-1}(x)h_0^{-1}$.
It follows that $x\in \Im(\varphi_{f_1})$ and $x\in h_0\Im(\varphi_{\ol{f_1}})h_0^{-1}$.
But if $t(f_1)=o(f_1)$ is hyperbolic this is not possible by Lemma \ref{lem:malnormal}
since the two boundary tori of $N_{t(f_1)}=N_{o(f_1)}$ corresponding to the edge $f_1$ are obviously different.
If $t(f_1)=o(f_1)$ is Seifert fibered, then we can similarly exclude this case by appealing to Lemma \ref{lem:sfsmalnormal} and Theorem \ref{thm:adjsfs}.

\noindent \emph{Case 2:}  The vertex $o(f_{l})$ is hyperbolic.
It follows easily from (\ref{equ:or}) and Lemma \ref{lem:malnormal}
that the path (\ref{equ:pathnow}) is reduced. Since the path represents $x$ this implies in particular that $l=1$.
We thus reduced Case 2 to Case 1.

\noindent \emph{Case 3:}  The vertex $o(f_{l})$ is Seifert fibered and $\varphi_{f_l}^{-1}(x)\not\in \ll c_{o(f_l)}\rr$.
Note that Lemma \ref{lem:sfsmalnormal} together  with Theorem \ref{thm:adjsfs} and  (\ref{equ:or}) implies that  the path (\ref{equ:pathnow}) is reduced, i.e. $l=1$.
We thus also reduced Case 3 to Case 1.

\noindent \emph{Case 4:}  The vertex $o(f_{l})$ is Seifert fibered, $\varphi_{f_l}^{-1}(x)\in \ll c_{o(f_l)}\rr$ and  $l>1$.
Note that by Theorem \ref{thm:sfs} (2) this implies that $h_{l-1}\varphi_{f_l}^{-1}(x)h_{l-1}^{-1}\in \Im(\varphi_{f_{l-1}})$.
We can thus represent $x$ by
\[ (h_0,f_1,\dots,f_{l-2},\,h_{l-2}\cdot \varphi_{f_{l-1}}^{-1}\big( h_{l-1}\varphi_{f_l}^{-1}(x)h_{l-1}^{-1}  \big)\cdot h_{l-2}^{-1},\, f_{l-2}^{-1},\dots,f_1^{-1},h_0^{-1}).\]
If $o(f_{l-1})$ is hyperbolic, then the argument of Case 2 immediately shows that $l=2$.
If $o(f_{l-1})$ is Seifert fibered, then it follows from Theorems \ref{thm:sfs} and \ref{thm:adjsfs} and from Lemma \ref{lem:sfsmalnormal} (2)
that $h_{l-2}\cdot \varphi_{f_{l-1}}^{-1}\big( h_{l-1}\varphi_{f_l}^{-1}(x)h_{l-1}^{-1}  \big)\cdot h_{l-2}^{-1}
\not\in \ll c_{o(f_{l-1})}\rr$. The argument of Case 3  immediately shows that again $l=2$.

We now  showed that $l=2$, we thus see that $x$ equals
\[ h_{0}\cdot \varphi_{f_{l-1}}^{-1}\big( h_{l-1}\varphi_{f_l}^{-1}(x)h_{l-1}^{-1}  \big)\cdot h_{0}^{-1}.\]
If $o(f_1)=t(f_2)$ is hyperbolic, then $x\in \Im(\varphi_{f_2})$
and $x\in h_0\Im(\varphi_{\ol{f_1}})h_0^{-1}$. It follows from Lemma \ref{lem:malnormal}
that $f_1=\ol{f_2}$ and $h_0\in \Im(\varphi_{\ol{f_1}})$. If we change the base point to $o(f_2)=t(f_1)$
we see that $x$ is represented by $\varphi_{f_2}^{-1}(x)\in G_{o(f_2)}$ and $y$ is represented by $\varphi_{f_1}(h_0)h_1\in G_{o(f_2)}$.
If on the other hand $o(f_1)=t(f_2)$ is Seifert fibered, then it follows from Theorem \ref{thm:adjsfs} that $x\not\in \ll c_{t(f_2)}\rr$.
It now follows easily from Lemma \ref{lem:sfsmalnormal} that  $f_1=\ol{f_2}$ and $h_0\in \Im(\varphi_{\ol{f_1}})$. We conclude the argument as above.

We now turn to the case that $l(x)>0$. We claim that Conclusion (1) holds.
By Theorem \ref{thm:divisible} we can find $z\in \pi_1(N)$ which is indivisible and $n>0$ with $x=z^n$.
Without loss of generality  assume that $z$ is cyclically reduced.
We claim that $y$ is a power of $z$ as well.
We represent $z$ by a reduced loop $q=(g_0,e_1,g_1,\dots,e_k,g_k)$. We now consider the path $p*q^n*p^{-1}$ which is given by
\[ (h_0,f_1,h_1,\dots,f_l,h_l\cdot  g_0,e_1,g_1,\dots,e_k,g_k\cdot h_l^{-1},f_l^{-1},\dots,h_1^{-1},f_1^{-1},h_0^{-1}).\]
This loop has to be reduced since $l>0$ and therefore the loop is longer than the loop $q^n$ which represents the same element. We conclude that one of the following conditions hold:
\bn
\item $f_l=\ol{e_1}$ and $h_lg_0\in \Im(\varphi_{f_l})$, or
\item $e_k=f_l$ and $g_kh_l^{-1}\in \Im(\varphi_{e_k})$.
\en
Note though that not both conclusions can hold, otherwise $x$ would not be cyclically reduced.
Now suppose that (1) holds and (2) does not hold. A straightforward induction argument now shows that $p=p'*q^{-1}$ for some reduced path $p'$.
On the other hand, if (2) holds and (1) does not hold, then a straightforward induction argument shows that $p=q^{-1}*p'$ for some reduced path $p'$.

\begin{claim}
If $l(p')=0$, then $p'$ represents the trivial element.
\end{claim}

If $l(p')=0$, then we denote by $y'$ the element represented by $p'$.  Suppose that $y'$ is non-trivial.
 In that case we have $y'x^n(y')^{-1}=x^n$ for any $n$,
in particular $x^ny'x^{-n}=y'$. It follows from the discussion of Cases 1, 2, 3 and 4 above that $l(x^n)\leq 2$
for any $n$. Since $x$ is cyclically reduced and $l(x)>0$ this case can not occur.
This concludes the proof of the claim.

If $p'$ represents the trivial element we are clearly done. If not, then $l(p')>0$
and we can do an induction argument on the length of $p'$ to  show that $y$ is in fact a power of $z$.
\end{proof}

\subsection{Malnormality of peripheral subgroups}

Using the methods of the proof of Theorem \ref{thm:commute}
we can now also prove the following theorem which was first proved by de la Harpe and Weber \cite{HW11}.

\begin{theorem}
Let $N$ be a compact, orientable, irreducible 3-manifold  with toroidal boundary and $S$ a boundary component.
If the JSJ component which contains $S$ is hyperbolic, then $\pi_1(S)\subset \pi_1(N)$ is malnormal.
\end{theorem}

\begin{proof}
Let $N$ be a compact, orientable, irreducible 3-manifold  with toroidal boundary and $S$ a boundary component.
 We denote by $\GG=\GG(N)$ the corresponding JSJ graph with vertex set $V$ and edge set $E$. 
Suppose that the JSJ component $N_w$ which contains $S$ is hyperbolic.
Now let $x\in \pi_1(S)$ and $g\in \pi_1(N)\sm \pi_1(S)$.

We pick a base point on $S$.  We  represent $g$ by a reduced loop
$p=(h_0,f_1,h_1,\dots,f_{l-1},h_{l-1},f_l,h_l)$ based at $w$.
If $l=0$, then $g\in \pi_1(N_w)$, but since $\pi_1(S)\subset \pi_1(N_w)$ is malnormal
by Lemma \ref{lem:malnormal} (1) it follows that $gxg^{-1}\not\in \pi_1(S)$.
Now suppose that $l>0$. We consider the path
\[ p*x*p^{-1}=(h_0,f_1,h_1,\dots,f_l,h_lxh_l^{-1},f_l^{-1},\dots,h_1^{-1},f_1^{-1},h_0^{-1}).\]
This path is reduced if and only if $x \in \Im(\varphi_{f_l})$.
But $\Im(\varphi_{f_l})$ is the image of a boundary torus in $N_w$ distinct from $S$.
It now follows from Lemma \ref{lem:malnormal} (2) that  $h_lxh_l^{-1} \not\in \Im(\varphi_{f_l})$.
We conclude that the path $p*x*p^{-1}$ is reduced, i.e. $gxg^{-1}$ does not lie in $\pi_1(N_w)$, let alone in $\pi_1(S)$.
\end{proof}

\subsection{Proof of Theorem \ref{thm:centgen}}

For the reader's convenience we recall the statement of Theorem \ref{thm:centgen}.

\begin{theorem}
Let $N$ be a  3-manifold. We write $\pi=\pi_1(N)$.
Let $g\in \pi$. If $C_{\pi}(g)$ is non-cyclic, then one of the following holds:
\bn
\item there exists a JSJ torus or a boundary torus $T$ and $h\in \pi$ such that $g\in h\pi_1(T)h^{-1}$ and such that
\[ C_{\pi}(g)=h\pi_1(T)h^{-1},\]
\item there exists a Seifert fibered component $M$ and $h\in \pi$ such that
 $g\in h\pi_1(M)h^{-1}$ and such that
\[ C_{\pi}(g)=hC_{\pi_1(M)}(h^{-1}gh)h^{-1}.\]
\en
\end{theorem}

\begin{proof}
Let $N$ be a  3-manifold and let $g\in \pi=\pi_1(N)$. If for any $h\in C_\pi(g)$ the group generated by $g$ and $h$ is cyclic,
then either $C_\pi(g)$ is cyclic, or $g$ is infinitely divisible. Since the former case is excluded by
 Theorem \ref{thm:divisible} the latter case has to hold.

Now suppose that $C_\pi(g)$ is not cyclic and suppose that there exist an $h\in C_\pi(g)$ such that the group generated by $g$ and $h$ is not cyclic.
It follows from Theorem \ref{thm:commute} that one of the following three cases occurs:
 \bn
 \item there exists a JSJ torus $T$ such that $g$ lies in a conjugate of $\pi_1(T)\subset \pi_1(N)$,
 \item  there exists a Seifert fibered component $M$ of the JSJ decomposition  such that $g$ lies in a conjugate of $\pi_1(M)\subset \pi_1(N)$,
\en
First suppose there exists a JSJ torus $T$ such that $g$ lies in a conjugate of $\pi_1(T)\subset \pi_1(N)$.
Without loss of generality we can assume that $g\in \pi_1(T)$. We first consider the case that the two JSJ components abutting $T$ are different.
We denote these two components by $M_1$ and $M_2$. By Proposition \ref{prop:hypbasics} (5) the following claim implies the theorem in this case.

\begin{claim}
There exists an $i\in \{1,2\}$ such that
\[ C_{\pi}(g)=C_{\pi_1(M_i)}(g).\]
\end{claim}

Let $h\in C_\pi(g)$. It follows easily from the proof of Theorem \ref{thm:commute} that
either $h\in \pi_1(M_1)$ or $h\in \pi_1(M_2)$. If  $M_1$ is hyperbolic, then it follows from Lemma \ref{lem:hypdivisible} and from Proposition \ref{prop:hypbasics} (5)
that $h\in \pi_1(T)$. It follows that $ C_{\pi}(g)=C_{\pi_1(M_2)}(g)$. Similarly we deal with the case that $M_2$ is hyperbolic.
Finally assume that $M_1$ and $M_2$ are Seifert fibered. We denote by $c_1$ and $c_2$ the regular fibers of $M_1$ and $M_2$.
If $g$ is not a power of $c_1$, then it  follows from Lemma \ref{lem:centbdy} that $C_{\pi}(g)=C_{\pi_1(M_2)}(g)$, similarly if $g$ is not a power of $c_2$.
Recall that $c_1$ and $c_2$ are indivisible in $\pi_1(T)$ and that by Theorem \ref{thm:adjsfs} we have $c_1\ne c_2^{\pm 1}$.
It follows that $g$ is either not a power of $c_1$ or not a power of $c_2$.

The case that the torus is non-separating can be dealt with similarly. We leave this to the reader.
Also, if  there exists a Seifert fibered component $M$ of the JSJ decomposition  such that $g$ lies in a conjugate of $\pi_1(M)\subset \pi_1(N)$
and such that $g$ does not lie in the image of a boundary torus, then it follows  easily from the proof of Theorem \ref{thm:commute} that
\[ C_{\pi}(g)=C_{\pi_1(M)}(g).\]

\end{proof}

%



\begin{thebibliography}{asd}


\bibitem{AFW11}
M. Aschenbrenner, S. Friedl and H. Wilton, {\em 3-manifold groups}, Preprint (2011)

\bibitem{Ba93}
H. Bass, {\it Covering theory for graphs of groups,} J. Pure Appl. Algebra {\bf 89} (1993), no. 1-2, $3$-47.


\bibitem{Co89}
D. E. Cohen, {\it Combinatorial Group Theory: A Topological Approach,} London Mathematical Society Student Texts, vol. 14, Cambridge University Press, Cambridge, 1989.

\bibitem{GH75}
C. McA. Gordon and W. Heil, {\em
Cyclic normal subgroups of fundamental groups of 3-manifolds},
Topology 14, 305-309 (1975).


\bibitem{HP07}
P. de la Harpe and J.-P. Pr\'eaux, {\em Groupes fondamentaux des vari\'et\'es de dimension 3 et alg\'ebres d'op\'erateurs}, Annales Fac. Sciences Toulouse, Math., S\'er. 6, Vol 16:3 (2007) 561-589


\bibitem{HW11}
P. de la Harpe and C. Weber, {\em On malnormal peripheral subgroups in fundamental groups of 3-manifolds},
in preparation (2011)

\bibitem{He87}
J. Hempel, {\em Residual finiteness for $3$-manifolds}, Combinatorial group theory and topology (Alta, Utah, 1984), 379--396, Ann. of Math. Stud., 111, Princeton Univ. Press, Princeton, NJ, 1987



\bibitem{JS79}
W. Jaco and P. Shalen, {\em Seifert fibered spaces in $3$-manifolds}, Mem. Amer. Math. Soc. 21 (1979), no. 220.


\bibitem{Jo79}
K. Johannson, {\em Homotopy equivalences of 3-manifolds with boundaries}, Lecture Notes in Mathematics, 761. Springer, Berlin, 1979.

\bibitem{Ma07}
A. Marden, {\em Outer circles.
An introduction to hyperbolic 3-manifolds},
Cambridge University Press, Cambridge, 2007.

\bibitem{Or72}
P. Orlik, {\em Seifert manifolds}, Lecture Notes in Mathematics, vol. 291, Springer-Verlag (1972).


\bibitem{OVZ67}
P. Orlik, E. Vogt and H. Zieschang, {\em Zur Topologie gefaserter dreidimensionaler Mannigfaltigkeiten},
Topology, 6 (1967), 49-64.


\bibitem{Sc83}
P.  Scott, {\em The geometries of $3$-manifolds}, Bull. London Math. Soc.
15 (1983), no. 5, 401--487

%

\bibitem{Se80} J.-P. Serre, {\em Trees}, Springer-Verlag, Berlin-New York, 1980.


\bibitem{Sh02}
P. Shalen, {\em Representations of 3--manifold groups}, Handbook of Geometric Topology, pp. 955--1044 (2002)


\bibitem{Th79}
W.  Thurston, {\em
The geometry and topology of 3-manifolds},
Princeton Lecture Notes (1979), available at\\
\texttt{http://www.msri.org/publications/books/gt3m/}



\bibitem{To78}
J.  Tollefson, {\em  Involutions of Seifert fiber spaces}, Pacific J. Math.
74 (1978), no. 2, 519--529.


\bibitem{Wa68}
F. Waldhausen, {\em On irreducible 3-manifolds which are sufficiently large},
Ann. Math. (2) 87, 56-88 (1968).



\end{thebibliography}
\end{document}